\documentclass[12pt]{article}
\usepackage{amsthm,amsmath,amssymb}
\usepackage{comment,fullpage}
\usepackage{amsfonts,amsmath,amsthm,amssymb,comment}
\usepackage{authblk}
\usepackage{graphicx}
\usepackage{mathrsfs}
\usepackage{bm}
\newtheorem{thm}{Theorem}[section]
\newtheorem{prop}[thm]{Proposition}

\newtheorem{corollary}[thm]{Corollary}

\newtheorem{example}[thm]{Example}

\def\ZZ{{\mathbb Z}}

\begin{document}

\title{\bf Biembedding Steiner Triple Systems and $n-$cycle Systems on Orientable Surfaces}

\author{Jeffrey H.\ Dinitz}
\author{Amelia R. W. Mattern}
\affil{Department of Mathematics and Statistics\\
University of Vermont\\
Burlington, VT 05405
U.S.A.}

\date{}

\maketitle
\begin{center}\today \end{center}

\begin{center}
{\em Dedicated to the memory of Dan Archdeacon, our friend, colleague, and teacher.}
\end{center}
\medskip

\begin{abstract}
In 2015, Archdeacon introduced the notion of Heffter arrays and showed the connection between Heffter arrays and biembedding  $m-$cycle and an $n-$cycle systems on a surface.  In this paper we exploit this connection and prove that for every $n \geq 3$ there  exists an orientable embedding of the complete graph on $6n+1$ vertices with each edge on both a $3$-cycle and an $n$-cycle. We also give an analogous (but partial) result for biembedding a $5$-cycle system and an $n$-cycle system.

\end{abstract}

\section{Introduction and definitions}
An embedding of the complete graph on $v$ vertices with each edge on both a $m$-cycle and an $n$-cycle is termed a {\em biembedding}. A biembedding is necessarily 2-colorable with the faces that are $m$-cycles receiving one color while those faces that are $n$-cycles receiving the other color. So each pair of vertices occur together in exactly one $m$-cycle and one $n$-cycle.  A {\em k-cycle system} on $v$ points is a collection of simple $k$-cycles with the property that any pair of points appears in a unique $k$-cycle.  Hence a biembedding is a simultaneous embedding of an $m$-cycle system and an $n$-cycle system on $v$ points.  In this paper we will specifically consider the case of biembeddings of 3-cycle systems (Steiner triple systems) and $n-$cycle systems where both of these systems are on $6n+1$ points.

There has been previous work done in the area of biembedding cycle systems, specifically Steiner triple systems. In 2004, both Bennet, Grannell, and Griggs \cite{B04} and Grannell and Korzhik \cite{G04} published papers on {\em nonorientable} biembeddings of pairs of Steiner triple systems. In \cite {G0915} the eighty Steiner triple systems of order 15 were also proven to have {\em orientable} biembeddings. In addition, Granell and Koorchik \cite{G09M} gave methods to construct orientable biembeddings of two cyclic Steiner triple systems from current assignments on M{\"o}bius ladder graphs. Brown \cite{B10} constructed a class of biembeddings where one face is a triangle and one face is a quadrilateral. 
Recently, Forbes, Griggs, Psomas, and {\v S}ir{\'a}{\v n} \cite{F14} proved the existence of biembeddings of pairs of Steiner triple systems in orientable pseudosurfaces with one pinch point, Griggs, Psomas and {\v S}ir{\'a}{\v n} \cite{GPS.14} presented a uniform framework for biembedding Steiner triple systems obtained from the Bose construction in both orientable and nonorientable surfaces,
 and McCourt \cite{M14} gave nonorientable biembeddings for the complete graph on $n$ vertices with a Steiner triple system of order $n$ and a Hamiltonian cycle for all $n \equiv 3 \pmod {36}$ with $n \geq 39$.

In 2015, Archdeacon \cite{A14} presented a framework for biembedding $m-$cycle and $n-$cycle systems on $v$ points on a surface for general $m$ and $n$.  It involved the use of so-called Heffter arrays and is quite general in nature, working in both the orientable and nonorientable case as well as for many possible values of $v$ for fixed $m$ and $n$. This is the first paper to explicitly use these Heffter arrays for biembedding purposes (they are actually of some interest in their own right).  In this paper we consider basically the smallest (and tightest) case for which this method works, namely we will prove that for every $n\geq 3$ there exists a biembedding of a Steiner triple system and an $n$-cycle system on $6n+1$ points.
We begin with the definitions of Heffter systems and Heffter arrays from \cite{A14}.  Some of the definitions in \cite{A14} are  more general, but these  suffice for our purpose.

Let $\mathbb Z_r$ be the cyclic group of odd order $r$ whose elements are denoted 0 and $\pm i$ where $i = 1,2,...,\frac{r-1}{2}$. A \textit{half-set} $L \subseteq \mathbb{Z}_r$ has $\frac{r-1}{2}$ nonzero elements and contains exactly one of $\{x, -x\}$ for each such pair. A \textit{Heffter system} $D(r,k)$ is a partition of $L$ into parts of size $k$ such that the sum of the elements in each part equals 0 modulo $r$. 
Note that a Heffter system $D(n,3)$ provides a solution to Heffter's first difference problem 
(see \cite{handbook-skolem}) and hence provides the base blocks for a cyclic Steiner triple system on $n$ points.

Two Heffter systems, $D_1 = D(2mn + 1, n)$ and $D_2 = D(2mn + 1, m)$, on the same half-set, $L$, are \textit{orthogonal} if each part (of size $n$) in $D_1$ intersects each part (of size $m$) in $D_2$ in a single element. A \textit{Heffter array} $H(m,n)$ is an $m \times n$ array whose rows form a $D(2mn +1, n)$, call it $D_1$, and whose columns form a $D(2mn+1, m)$, call it $D_2$. Furthermore, since each cell $a_{i,j}$ contains the shared element in the $i^{th}$ part of $D_1$ and the $j^{th}$ part of $D_2$, these row and column Heffter systems are orthogonal. So an $H(m,n)$ is equivalent to a pair of orthogonal Heffter systems. In Example \ref{Heff} we give orthogonal Heffter systems $D_1 = D(31,5)$ and $D_2 = D(31,3)$ along with the resulting Heffter array $H(3,5)$.  Note that the elements occurring in the array form a half set of $Z_{31}.$

\begin{example}
A Heffter system $D_1 = D(31,5)$ and a Heffter system $D_2 = D(31,3)$: 
$$\begin{array}{l}
D_1 = \{\{6,7,-10,-4,1\}, \{-9,5,2,-11,13\}, \{3,-12,8,15,-14\}\},  \\
D_2 = \{\{6,-9,3\}, \{7,5,-12\}, \{-10,2,8\}, \{-4,-11,15\}, \{1,13,-14\}\}. 
\end{array}
$$
\newpage
The resulting Heffter array $H(3,5)$:
\begin{center}
$\begin{bmatrix}
 6 & 7 & -10 & -4 & 1 \\
-9 & 5 & 2 & -11 & 13 \\
3 & -12 & 8 & 15 & -14
\end{bmatrix}.$
\end{center}
\label{Heff}
\end{example}
\vspace*{.1in}

Let $A$ be a subset of $\mathbb Z_{2mn+1} \setminus \{0\}$. 
Let $(a_1,...,a_k)$ be a cyclic ordering of the elements in $A$ and let $s_i = \sum_{j=1}^i a_j \pmod {2mn+1}$ be the $i^{th}$ partial sum. The ordering is \textit{simple} if $s_i \not = s_j$ for $i \not = j$. A Heffter system $D(2mn+1,k)$ is \textit{simple} if and only if each part has a simple ordering. Further, a Heffter array $H(m,n)$ is {\em simple} if and only if its row and column Heffter systems are simple.

In the next section we will give the connection between Heffter arrays and biembeddings of the complete graph.  In Section \ref{section3} we use this connection to show that for all $n\geq 3$ there exists a biembedding of the complete graph on $6n+1$ such that each edge is on a simple face of size $n$ and a face  of size 3.  In Section \ref{section4} we discuss biembeddings of the complete graph on $10n+1$ for $3\leq n \leq 100$ such that each edge is on a simple face of size $n$ and a simple face  of size 5.

\section{Heffter arrays and biembeddings}
In this section we establish the relationship between Heffter arrays and biembeddings. 
The following proposition  from \cite{A14} describes the connection between Heffter systems and $k$-cycle systems.

\begin{prop}\label{prop2.1}
\cite{A14} The existence of a simple Heffter system $D(v,k)$ implies the existence of a simple $k$-cycle system decomposition of the edges $E(K_{v})$. Furthermore, the resulting $k$-cycle system is cyclic.
\end{prop}

Let $D_1 = D(2mn+1,m)$ and $D_2 = D(2mn+1,n)$ be two orthogonal Heffter systems with orderings $\omega_1$ and $\omega_2$ respectively. The orderings are \textit{compatible} if their composition $\omega_1 \circ \omega_2$ is a cyclic permutation on the half-set. The following theorem relates $m \times n$ Heffter arrays with compatible simple orderings on the rows and columns to orientable biembeddings of $K_{2mn+1}$.

\begin{thm}
\cite{A14} Given a Heffter array $H(m,n)$ with simple compatible orderings $\omega_r$ on $D(2mn+1,n)$ and $\omega_c$ on $D(2mn+1,n)$, there exists an embedding of $K_{2mn+1}$ on an orientable surface such that every edge is on a simple cycle face of size $m$ and a simple cycle face of size $n$.
\label{bigthm_corr}
\end{thm}

In the following theorem we prove that if $m$ and $n$ are not both even, then there exist orderings $\omega_r$ and $\omega_c$ of the row and column Heffter systems, respectively, that are compatible. Archdeacon knew this result, however it is not included in \cite{A14}.

\begin{thm}
Let $H$ be a $m\times n$ Heffter array where at least one of $m$ and $n$ is odd. Then there exist compatible orderings, $\omega_r$ and $\omega_c$ on the row and column Heffter systems.
\label{compatible}
\end{thm}

\begin{proof}  
Without loss of generality we assume that the number of columns  $H$ is odd, so say $n=2t+1$ for some integer $t$.
Let $H = (h_{ij})$ be a $m \times n$ Heffter array. We first define $\omega_r$, the ordering of the row Heffter system,  as $\omega_r = (h_{11}, h_{12}, \ldots,h_{1n}) 
(h_{21}, h_{22}, \ldots,h_{2n}) \ldots (h_{m,1}, h_{m,2}, \ldots,h_{m,n})$.  This ordering says that each row in the row Heffter system of $H$ is ordered cyclically from left to right.  We next order the columns.
For  $1\leq c \leq t+1$  the ordering for column $c$ is   $(h_{1,c},h_{2,c}, \ldots ,h_{m,c})$ (this is basically top to bottom cyclically)  and for $t+2\leq c \leq n$  the ordering for column $c$ is   $(h_{m,c},h_{m-1,c}, \ldots ,h_{1,c})$ (bottom to top, cyclically). So considering the composition $\omega_r \circ \omega_c$ we have that 

$$\omega_r \circ \omega_c(h_{i,j}) = \left\{ 
\begin{array}{ll}
h_{i+1,j+1} & \mbox{if } 1\leq j \leq t+1 \\ 
h_{i-1,j+1}& \mbox{if } t+2\leq j \leq n 
\end{array}\right. 
$$

\noindent 
where all first subscripts are written as elements from $\{1,2, \ldots ,m\}$ reduced modulo $m$ and all second subscripts are written as elements from $\{1,2, \ldots ,n\}$ reduced modulo $n$.

So by construction, starting at any cell in column 1 we see that $\omega_r \circ \omega_c$ moves cyclically from left to right and goes ``down'' $t+1$ times and ``up'' $t$ times. Hence for any $r$ and $c$, given an occurrence of $h_{r,c}$ in  $\omega_r \circ \omega_c$ the next occurrence of column $c$ in $\omega_r \circ \omega_c$ will be $h_{r+1,c}$.  It is now straightforward to  see that 
$$
\begin{array}{rl}
\omega_r \circ \omega_c = &(h_{1,1}, h_{2,2}, \ldots,h_{t+1,t+1}, h_{t,t+2}, \ldots h_{3,n}\\
&h_{2,1 }, h_{3,2} \ldots,h_{t+2,t+1}, h_{t+1,t+2},\ldots h_{4,n },\\
&h_{3,1 }, h_{4,2} \ldots,h_{t+3,t+1}, h_{t+2,t+2},\ldots h_{5,n },  \\
&\ \ \ \  \vdots\\
&h_{m,1 }, h_{m+1,2} \ldots,h_{m+t+1,t+1},\ldots h_{m,n }).
\end{array}
$$
 Hence we have that  $\omega_r \circ \omega_c$ is written as a single cycle on the half set and thus  $\omega_r$ and $\omega_c$ are compatible orderings.
\end{proof}

Now from  Theorems \ref{bigthm_corr} and \ref{compatible} we have the following theorem relating Heffter arrays and biembeddings.

\begin {thm}\label{heffter-biembed.thm}
Given a simple Heffter array $H(m,n)$  where at least one of $m$ and $n$ is odd, there exists an embedding of $K_{2mn+1}$ on an orientable surface such that every edge is on a simple cycle face of size $m$ and a simple cycle face of size $n$ 
\end{thm}

Restating the previous theorem in terms of biembeddings of cycle systems we have the following.

\begin {thm}
Given a simple Heffter array $H(m,n)$ with where at least one of $m$ and $n$ is odd, there exists an orientable biembedding of an $m-$cycle system and an $n-$cycle system both on $2mn+1$ points.
\end{thm}

\section{Constructing simple $H(3,n)$} \label{section3}

In this section we will construct simple $H(3,n)$ for all $n \geq 3$.  For each $n$, we will begin with the $3 \times n$ Heffter array already constructed in \cite{Ain} and will provide a reordering so that the resulting Heffter array is simple.   We first record  the existence result for $H(3,n)$.

\begin{thm}
\cite{Ain} There exists a $3 \times n$ Heffter array for all $n \geq 3$.
\label{3xn}
\end{thm}

We now restate Theorem \ref{heffter-biembed.thm} in the special case when there are 3 rows in the Heffter array.

\begin {corollary}\label{heffter-biembed.3xn}
If there exists a simple Heffter array $H(3,n)$, then there exists an embedding of $K_{6n+1}$ on an orientable surface such that every edge is on a simple cycle face of size $3$ and a simple cycle face of size $n$.

\end{corollary}

Suppose $H = (h_{ij})$ is any $3 \times n$ Heffter array. We first note that each column $c$ in $H$, $1\leq c\leq n$, is simple just using the natural  top-to-bottom ordering. 
Thus if we can  reorder each of the three rows so they have distinct partial sums, the array will be a simple.   In   Theorem \ref{reord2} below we will present a single reordering for each Heffter array that makes $\omega_r$ simple. In finding a {\em single} reordering which works for all three rows in $H$, we are actually rearranging the order of the columns without changing the elements which appear in the rows and columns. For notation, the {\em ordering} $(a_1, a_2, \ldots, a_n)$  denotes a reordering of the columns of $H$ so that in the resulting array $H',$ column $a_i$ of $H$ will appear in column $i$ of $H'$. In Example \ref{reord} we give the original array $H(3,8)$, the reordering $R$ for the rows, and the reordered array $H'(3,8)$.

\begin{example}
The original $3 \times 8$ Heffter array from \cite{Ain}:
$$H=\begin{bmatrix}
-13 & -11 & 6 & 3 &10 & -8 & 14 & -1 \\
 4 & -7 & 17 & 19 & 5 & -16 & -2 & -20 \\
 9 & 18 & -23 & -22 & -15 & 24 & -12 & 21
\end{bmatrix}.$$

Note that in row 1, $s_1 = s_6 = -13 \equiv 36 \pmod {49}$, and so $\omega_r$ is not simple.  Consider the reordering $R = (1,2,6,8,5,3,4,7)$. 
The reordered $3 \times 8$ Heffter array is:
$$H'=\begin{bmatrix}
-13 & -11 & -8 & -1 &10 & 6 & 3 & 14 \\
 4 & -7 & -16 & -20 & 5 & 17 & 19 & -2 \\
 9 & 18 & 24 & 21 & -15 & -23 & -22 & -12
\end{bmatrix}.$$

We list the partial sums for each row as their smallest positive residue modulo 49: \\

\hspace*{1.5in}
\begin{tabular}{l}
Row 1: \ \ $\{36, 25, 17, 16, 26, 32, 35, 0\},$ \\
Row 2: \ \ $\{4, 46, 30, 10, 15, 32, 2, 0\},$ \\
Row 3: \ \ $\{9,27,2,23,8,34,12,0\}.$ \end{tabular}

Since all the partial sums are distinct,  $\omega_r$ is simple and hence $H'(3,8)$ is a simple Heffter array.
\label{reord}
\end{example}

As the reader can see, the reordering of the columns results in a simultaneous reordering of the three rows resulting in a simple $H(3,n)$ where the elements in each row and column remain the same as in the original array. We handle two small cases in the next theorem.

\begin{thm}
There exist simple $H(3,3)$ and $H(3,4)$. 
\label{reord1}
\end{thm}

\begin{proof}

We present an $H(3,3)$ and an $H(3,4)$ from  \cite{Ain}.  It is easy to check that both are simple.

\begin{center}
$
\begin{array}{|c|c|c|} \hline
-8&-2&-9 \\ \hline
7&-3&-4 \\ \hline
1&5&-6 \\ \hline
\end{array}$ \hspace{.7in}
$
\begin{array}{|c|c|c|c|} \hline
1& 2& 3& -6 \\ \hline
8& -12& -7& 11 \\ \hline
-9& 10& 4& -5 \\ \hline
\end{array}$
\end{center}
\end{proof}

In our next theorem we will construct simple $H(3,n)$ for all $n \geq 5$. The cases are broken up modulo 8 and we will consider each individually. 
We will begin with the $3\times n$ Heffter array $H$ given in \cite{Ain} and reorder the columns. In all cases we let $H'$ be the $3 \times n$ Heffter array where the columns of $H$ have been reordered as given in each case. We will always write the partial sums as their lowest positive residue modulo $6n+1$. For the following theorem we introduce the following notation for intervals in $\ZZ$, let $[a,b] = \{a, a+1, a+2, \ldots, b-1, b\}\ \ \mbox{   and    } \ \ [a,b]_2 = \{a, a+2, a+4, \ldots, b-2, b\}.$

\begin{thm} There exist simple $3\times n$ Heffter arrays for all $n \geq 3$.
\label{reord2}
\end{thm}

\begin{proof} 
 When $n=3$ or 4 the result follows from Theorem \ref{reord1} above.  We now assume that $n\geq 5$.
We begin with $H$,  a $3 \times n$ Heffter array from \cite{Ain}. Let $R$ be the reordering  for each row. For each $i = 1,2,3$ define $P_i$ as the set of partial sums of row $i$. We will divide each $P_i$ into four subsets, $P_{i,1}$, $P_{i,2}$, $P_{i,3}$ and $P_{i,4}$,  based on a natural partition of the columns of $H$.  For each case modulo 8 we will present the original construction from \cite{Ain}, followed by the reordering $R$ and the subsets $P_{i,j}$.

\bigskip
\noindent
If $\bm{n \equiv 0 \pmod 8, n \geq 8}$:
The case of $n=8$ is given above in Example \ref{reord}.  For $n >8$, define $m = \frac{n-8}{8}$, so $n=8m+8$ and hence all the arithmetic in this case will be in $\ZZ_{48m+49}$. The first four columns are:
{\small $$A = \begin{bmatrix}
-12m-13 & -10m-11 & 4m+6 & 4m+3 \\
4m+4 & -8m-7 & 18m+17 & 18m+19  \\
8m+9 & 18m+18 & -22m-23 & -22m-22
\end{bmatrix}.$$}
For each $0 \leq r \leq 2m$ define
{\small
$$A_r = (-1)^r\begin{bmatrix}
(8m+r+10) & (-8m+2r-8) & (14m-r+14) & (-4m+2r-1) \\
 (8 m - 2 r + 5)  & (-16 m - r - 16)  & (-4m + 2r - 2) & (-18m - r - 20) \\
 (-16 m + r - 15)  & (24 m - r + 24)  & (-10m - r - 12) & (22 m - r + 21) 
 \end{bmatrix}.$$} 
Beginning with the matrix $A$, we add on the remaining $n-4$ columns by concatenating the $A_r$ arrays for each value of $r$ between $0$ and $2m$. So the final array will be:
 $$H= \begin{bmatrix}
 A & A_0 & A_1 & \cdots & A_{2m}
 \end{bmatrix}.$$

Let $R = (9,13,...,n-3;1,11,15,..,n-1;2,10,14,..,n-2;6,8,12,16,..., n,5,3,7,4)$ be a reordering of the rows. Note that we use semi-colons to designate the partitions of $P_i$ into its four subsets. So in this case, $P_{i,1}$ is the set of partial sums from row $i$ and columns $\{9,13,...,n-3\}$ in $H$, $P_{i,2}$ is the set of partial sums from row $i$ and columns $\{1,11,15,...,n-1\}$ from $H$, $P_{i,3}$ is the set of partial sums from row $i$ and columns $\{2,10,14,...,n-2\}$, and $P_{i,4}$ is the set of partial sums from row $i$ and columns $\{6,8,12,16,...,n,$ $5,3,7\}$. We must show that the partial sums of the rows of reordered array are all distinct.  To check this, we provide the following table where we give the ranges of the partial sums:
\renewcommand{\tabcolsep}{2pt}
\begin{center}
\begin{tabular}{|rcl|}
\hline
$R$ &= &$\{9,13,...,n-3;1,11,15,..,n-1;2,10,14,..,n-2;6,8,12,16,...,n,5,3,7,4\}$ \\
\hline
$P_{1,1}$  &= & $[39m+39, 40m+38] \cup [1,m]$ \\
\hline
$P_{1,2}$&= & $[36m+36, 37m+36] \cup [23m+23,24m+22]$ \\
\hline
$P_{1,3}$ &= &$[26m+25,28m+25]_2 \cup [32m+33, 34m+31]_2$ \\
\hline
$P_{1,4}$ &= & $[16m+16,18m+16]_2 \cup [18m+17, 20m+17]_2\cup \{26m+26,30m+32,0\}$ \\  \hline
\end{tabular}
\end{center}
One can count the number of sums in each of these partitions to show that all the sums must be distinct.  For example, in the first partition there will be $((n-3) -9)/4 = 2m-1$ partial sums. Since there are $2m-1$ elements in the range $[39m+39, 40m+38] \cup [1,m]$ it must be that all the partial sums are distinct.

We see that the ranges of the partial sums of elements in the first row are:

\noindent
$P_1 =\{0\}  \cup [1,m] \cup [16m+16,18m+16]_2  \cup [18m+17, 20m+17]_2 \cup [23m+23,24m+22]$ 
$\cup [26m+25,28m+25]_2 \cup \{26m+26,30m+32\} \cup [32m+33, 34m+31]_2 \cup [36m+36, 37m+36]$ 
$ \cup [39m+39, 40m+38]$. 

\medskip
\noindent
For the second row we get the following ranges:
\begin{center}
\begin{tabular}{|rcl|}

\hline
$P_{2,1}  $&= &$ [40m+46,42m+44]_2 \cup [46m+49,48m+47]_2$ \\
\hline
$P_{2,2}  $&= &$ [2m+4,4m+6]_2 \cup [4m+8,6m+4]_2$ \\
\hline
$P_{2,3} $&= &$ [12m+14,13m+13] \cup [43m+46,44m+45]$ \\
\hline
$P_{2,4}   $&= &$ [27m+30,28m+30] \cup [8m+10,9m+10] \cup \{16m+15,34m+32,30m+30,0\}$ \\
\hline
\end{tabular}
\end{center}

\noindent
Thus the ranges of the partial sums of elements in the second row are:

\noindent
$P_2 = \{0\} \cup [2m+4,4m+6]_2 \cup [4m+8,6m+4]_2 \cup [8m+10,9m+10] \cup [12m+14,13m+13] \cup \{16m + 15\} \cup [27m+30,28m+30] \cup \{30m+30\} \cup \{34m+32\} \cup [40m+46,42m+44]_2 \cup [43m+46,44m+45]
 \cup [46m+49,48m+47]_2$.\\

\noindent
For the third row we have:
\begin{center}
\begin{tabular}{|rcl|} \hline
$P_{3,1}  $&= &$  [1,m] \cup [15m+15,16m+14]$\\
\hline 
$P_{3,2}  $&= &$  [8m+9,9m+9] \cup [19m+22,20m+21]$ \\ 
\hline
$P_{3,3}  $&= &$  [2m+4,3m+3] \cup [25m+27,26m+27]$  \\
\hline
$P_{3,4} $&= &$   [m+2,2m+2] \cup [22m+22,23m+23] \cup \{6m+8,32m+34,0\}$ \\
\hline
\end{tabular}
\end{center}

\noindent
Thus the ranges of the partial sums of elements in the third row are:

\noindent
$P_3  =  \{0\} \cup  [1,m] \cup [m+2,2m+2] \cup [2m+4,3m+3] \cup \{6m+8\} \cup [8m+9,9m+9] \cup [15m+15,16m+14]
\cup [19m+22,20m+21] \cup [22m+22,23m+23] \cup [25m+27,26m+27] \cup \{32m+34\}.$ 

\medskip
Several things are worth noting.  
Notice that each partion of the partial sums covers two disjoint ranges of numbers (some sets contain a few extra numbers). For example, $P_{1,1}$ contains the range $39m+39$ to $40m+38$ and the range $1$ to $m$. This is by design.  Also, within these ranges the sets of partial sums either contain every number in the range, or every other number in the range. Note that any overlap of the sets of partial sums occurs with one set covering the odds and one covering the evens. Therefore one can check by looking at the ranges that the partial sums $P_i$ in each row are distinct. Similar arguments can be used for each case of $n$ modulo $8$. In all subsequent cases we will provide the reader with the original construction, the reordering, and a table of the partial sums. For further details see \cite{M15}.

\bigskip
\noindent
If $\bm{n \equiv 1 \pmod 8, n \geq 9}$: Here $m = \frac{n-9}{8}$ and note that
 all the arithmetic in this case will be in $\ZZ_{48m+55}$. The first five columns are: 
{\small
$$A = \begin{bmatrix}
8m + 7, & 10m + 12 & 16m + 18 &  4m+6 & 4m+3 \\
8m + 10 & 8m + 9 & -12m - 14 & -22m - 26 & 18m + 22  \\
-16m - 17 & -18m - 21&-4m - 4 & 18m + 20 & -22m - 25
\end{bmatrix}.$$
For each $0 \leq r \leq 2m$ define
}
{\small
\begin{center}
$A_r = (-1)^r\begin{bmatrix}
(-8m + 2r - 5), & (-10m - r - 13) & (-24m + r - 27) & (-4m + 2r - 1) \\
(16m - r + 16) & (-4m + 2r - 2) & (8m - 2r + 8) & (-18m - r - 23) \\
(-8m - r - 11) & (14m - r + 15) & (16m + r + 19) & (22m - r + 24)
 \end{bmatrix}.$ \end{center}
}

To construct $H$, begin with $A$ and add on the remaining $n-5$ columns by concatenating the $A_r$ arrays for each value of $r$ between $0$ and $2m$.
Let the reordering be 

\noindent \begin{center}
$R =(8,12,16,...,n-1;3,7,11,15,...,n-2;5,6,10,14,...,n-3;1,9,13,17,...,n,2,4)$. 
\end{center}
 Then
\begin{center}
\begin{tabular}{|rcl|} \hline
$P_{1,1} $&=&$ [24m + 28, 25m +28] \cup [47m + 55, 48m +54]$ \\
\hline
$P_{1,2} $&=&$ [41m+46 , 42m + 46] \cup [30m +33, 31m +33]$ \\
\hline
$P_{1,3} $&=&$ [32m+36 ,34m + 36]_2 \cup [26m + 31, 28m+31]_2$ \\
\hline
$P_{1,4} $&=&$ [34m+38, 36m + 38]_2 \cup [32m+37,34m+37]_2 \cup \{44m+49,0\}$  \\
\hline \end{tabular} \end{center}

\noindent
Thus
$P_1 = \{0\} \cup [24m + 28, 25m +28] \cup [26m + 31, 28m+31]_2 \cup [30m +33, 31m +33] \cup [32m+36 ,34m + 36]_2\cup [32m+37,34m+37]_2 \cup [34m+38, 36m + 38]_2  \cup [41m+46 , 42m + 46] \cup \{44m+49\}$ \\
$\cup [47m + 55, 48m +54]$.\\

\begin{center}
\begin{tabular}{|rcl|} \hline
$P_{2,1} $&=&$ [2,2m] \cup [6m+8,8m+8]$ \\
\hline
$P_{2,2} $&=&$ [38m + 47,40m+47]_2 \cup [40m+49, 42m+49]_2$ \\
\hline
$P_{2,3} $&=&$ [10m+14,11m+14] \cup [25m+30,26m+30]$ \\
\hline
$P_{2,4} $&=&$ [14m+17,15m+17] \cup [33m+40,34m+40] \cup \{22m+26,0\}$ \\
\hline
\end{tabular} \end{center}

\noindent
Thus $P_2 = \{0\} \cup [2,2m] \cup [6m+8,8m+8] \cup [10m+14,11m+14] \cup [14m+17,15m+17] \cup \{22m+26\} \cup
[25m+30,26m+30] \cup [33m+40,34m+40] \cup [38m + 47,40m+47]_2 \cup [40m+49, 42m+49]_2.$ \\

\begin{center}
\begin{tabular}{|rcl|} 
\hline
$P_{3,1} $&=&$ [47m+54, 48m+54] \cup [16m+19,17m+19]$ \\
\hline
$P_{3,2} $&=&$ [13m+15,14m+15] \cup [26m+30,27m+30]$ \\
\hline
$P_{3,3} $&=&$ [43m+49, 44m+49] \cup [4m+5,5m+5]$  \\
\hline
$P_{3,4} $&=&$ [27m+32, 28m+32] \cup [1,m+1] \cup \{30m+35,0\}$ \\
\hline
\end{tabular} \end{center}

\noindent
So $P_3 = \{0\} \cup [1,m+1] \cup [4m+5,5m+5] \cup [13m+15,14m+15] \cup  [16m+19,17m+19]$ \\
$ \cup [26m+30,27m+30] \cup [27m+32, 28m+32] \cup \{30m+35\} \cup [43m+49, 44m+49] \cup [47m+54, 48m+54].  $

\bigskip
\noindent
If $\bm{n \equiv 2 \pmod 8, n \geq 10}$: In this case $m = \frac{n-10}{8}$.
All the arithmetic in this case will be in $\ZZ_{48m+61}$.
 The first six columns are:
{\small
$$A = \begin{bmatrix}
24m + 30 & 16m + 21 & 10m + 13 & 8m + 8 & 4m + 5 & 8m + 9 \\
24m + 29 & -8m - 11 & -10m - 14 & 12m + 16 & 16m + 20 & 12m + 17 \\
2 & -8m - 10 & 1 & -20m - 24 & -20m - 25 & -20m - 26
\end{bmatrix}.$$}
For each $0 \leq r \leq 2m$ define 
{\small
$$A_r = (-1)^r\begin{bmatrix}
(-8m + 2r - 7) & (10m + r + 15) & (-22m + r - 27) & (-8m + 2r - 6) \\
(16m - r + 19) & (4m - 2r + 3) & (4m - 2r + 4) & (-16m - r - 22) \\
(-8m - r - 12)  & (-14m + r - 18) & (18m + r + 23) & (24m - r + 28)
 \end{bmatrix}.$$ }
 
To construct $H$, begin with $A$ and add on the remaining $n-6$ columns by concatenating the $A_r$ arrays for each value of $r$ between $0$ and $2m$. Let the reordering be 

\begin{center}
$R = (10,14,...,n;n-3,n-7,...,7;4,6,8,12,...,n-2;5,9,13,...,n-1,2,3,1).$
\end{center}
Then
\begin{center}
\begin{tabular}{|rcl|}
\hline
$P_{1,1} $&=&$ [40m+55,42m+55]_2 \cup [46m+61,48m+59]_2$ \\
\hline
$P_{1,2} $&=&$  [36m+48,38m+48]_2 \cup [42m+57,44m+55]_2 \cup \{44m+56\}$ \\
\hline
$P_{1,3} $&=&$  [3m+4,4m+4] \cup [14m+19,15m+19]$ \\
\hline
$P_{1,4} $&=&$  [18m+24,19m+24] \cup [45m+58,46m+58] \cup \{14m+18, 24m+31,0\}$ \\ \hline
\end{tabular}\end{center}

\noindent
Thus $P_1 = \{0\} \cup [3m+4,4m+4] \cup \{14m+18\} \cup [14m+19,15m+19] \cup [18m+24,19m+24] \cup \{24m+31\}
 \cup [36m+48,38m+48]_2 \cup [40m+55,42m+55]_2 \cup [42m+57,44m+55]_2 \cup \{44m+56\} \cup$ \\
$ [45m+58,46m+58] \cup [46m+61,48m+59]_2$. \\

\begin{center}
\begin{tabular}{|rcl|}
\hline
$P_{2,1} $&=&$ [1,m] \cup [31m+39,32m+39] $\\
\hline
$P_{2,2} $&=&$ [45m+58,46m+58] \cup [30m+39,31m+38] \cup \{10m+13\} $\\
\hline
$P_{2,3} $&=&$ [22m+30,24m+30]_2 \cup [24m+33,26m+33]_2 $ \\
\hline
$P_{2,4} $&=&$ [40m+53,42m+53]_2 \cup [42m+57,44m+57]_2 \cup \{34m+46, 24m+32,0\}$ \\ \hline
\end{tabular}\end{center}
\noindent
Thus
$P_2 = \{0\} \cup [1,m] \cup \{10m+13\} \cup [22m+30,24m+30]_2 \cup \{24m+32\} \cup [24m+33,26m+33]_2
\cup [30m+39,31m+38] \cup [31m+39,32m+39] \cup \{34m+46\} \cup [40m+53,42m+53]_2$ \\
$ \cup [42m+57,44m+57]_2 \cup [45m+58,46m+58] $.

\begin{center}
\begin{tabular}{|rcl|}
\hline
$P_{3,1} $&=&$ [1,m] \cup [23m+28,24m+28]$\\ 
\hline
$P_{3,2} $&=&$ [13m+16,14m+16] \cup [22m+28,23m+27] \cup \{42m+53\}$\\
\hline
$P_{3,3} $&=&$ [8m+9,9m+9] \cup [21m+27,22m+27] $ \\
\hline
$P_{3,4} $&=&$ [7m+7,8m+7] \cup [36m+45,37m+45] \cup \{48m+58, 48m+59,0\}$ \\
\hline
\end{tabular}\end{center}
\noindent
So $P_3 = \{0\} \cup  [1,m] \cup [7m+7,8m+7] \cup [8m+9,9m+9] \cup [13m+16,14m+16] \cup [21m+27,22m+27] 
\cup [22m+28,23m+27] \cup [23m+28,24m+28] \cup [36m+45,37m+45] \cup \{42m+53\}
 \cup \{48m+58, 48m+59\}.$

\bigskip
\noindent
If $\bm{n \equiv 3 \pmod 8, n \geq 11}$: Define $m = \frac{n-11}{8}$.
All the arithmetic in this case will be in $\ZZ_{48m+67}$.
 The first seven columns are:
{\small  
$$A = \begin{bmatrix}
24m + 33 & 8m + 11 & 8m + 13 & 4m + 6 & 1 & -12m - 17 & 8m + 10 \\
24m + 32 & -16m - 23 & -12m - 18 & 10m + 15 & 20m + 27 & -8m - 9 & 14m + 20 \\
2 & 8m + 12 & 4m + 5 & -14m - 21 & -20m - 28 & 20m + 26 & -22m - 30
\end{bmatrix}.$$ }
For each $0 \leq r \leq 2m$ define
{\small  
$$A_r = (-1)^r\begin{bmatrix}
(-16m + r - 22) & (24m - r + 31) & (4m - 2r + 4) & (-4m + 2r - 3) \\
(8m - 2r + 8) & (-8m + 2r - 7)  & (-22m + r - 29) & (-10m - r - 16) \\
(8m + r + 14)  & (-16m - r - 24) & (18m + r + 25) &  (14m - r + 19)
 \end{bmatrix}.$$}

To construct $H$, begin with $A$ and add on the remaining $n-7$ columns by concatenating the $A_r$ arrays for each value of $r$ between $0$ and $2m$. Let the reordering be 

\begin{center}
$R = (9,13,...,n-2;8,12,...,n-3;1,11,15,..,n;6,7,10,14,..., n-1,5, 2, 3, 4).$ 
\end{center}
Then
\begin{center}
\begin{tabular}{|rcl|}
\hline
$P_{1,1} $&=&$ [1,m] \cup [23m+31,24m+31] $ \\
\hline
$P_{1,2} $&=&$ [7m+9,8m+9] \cup [22m+31,23m+30]$ \\
\hline
$P_{1,3} $&=&$ [28m+39,30m+39]_2 \cup [30m+42,32m+42]_2 $\\
\hline
$P_{1,4} $&=&$ \{18m+22\} \cup [26m+32,28m+32]_2 \cup [28m+38,30m+36]_2 $\\
&& $ \cup
 \{28m+36,28m+37,36m+48,44m+61,0\}$\\
\hline
\end{tabular}\end{center}
\noindent
Thus
$P_1 = \{0\} \cup [1,m] \cup [7m+9,8m+9] \cup \{18m+22\} \cup [22m+31,23m+30] \cup [23m+31,24m+31]
\cup [26m+32,28m+32]_2 \cup \{28m+36, 28m+37\} \cup [28m+38,30m+36]_2 \cup [28m+39,30m+39]_2
\cup [30m+42,32m+42]_2 \cup \{36m+48, 44m+61\}. $

\begin{center}
\begin{tabular}{|rcl|}\hline
$P_{2,1} $&=&$ [40m+60,42m+60]_2 \cup [46m+67,48m+65]_2 $ \\
\hline
$P_{2,2} $&=&$ [42m+62,44m+60]_2 \cup [1,2m+1]_2 $ \\
\hline
$P_{2,3} $&=&$ [13m+17,14m+17] \cup [24m+33,25m+33] $\\
\hline
$P_{2,4} $&=&$ \{5m+8\} \cup [45m+66,46m+66] \cup [18m+28,19m+28]$\\
 && $ \cup \{18m+26,2m+3,38m+52,0\}$ \\
\hline
\end{tabular}\end{center}
\noindent
Thus
$P_2 = \{0\} \cup [1,2m+1]_2 \cup \{2m+3, 5m+8\} \cup [13m+17,14m+17] \cup \{18m+26\} \cup [18m+28,19m+28]
\cup [24m+33,25m+33] \cup \{38m+52\} \cup [40m+60,42m+60]_2 \cup [42m+62,44m+60]_2 \cup
[45m+66,46m+66] \cup  [46m+67,48m+65]_2. $

\begin{center}
\begin{tabular}{|rcl|} \hline
$P_{3,1} $&=&$ [1,m] \cup [31m+43,32m+43]$\\
\hline
$P_{3,2} $&=&$ [30m+43,31m+42] \cup [39m+57,40m+57]$  \\
\hline
$P_{3,3} $&=&$ [40m+59,41m+59] \cup [5m+11,6m+11] $\\
\hline
$P_{3,4} $&=&$ \{25m+37\} \cup [2m+7,3m+7] \cup [21m+32,22m+32]$ \\
&& $\cup \{2m+4,10m+16,14m+21,0\}$ \\
\hline
\end{tabular}\end{center}
\noindent
So $P_3 = \{0\} \cup [1,m] \cup  \{2m+4\} \cup [2m+7,3m+7] \cup [5m+11,6m+11] \cup \{10m+16, 14m+21\}
\cup [21m+32,22m+32] \cup \{25m+37\} \cup [30m+43,31m+42] \cup [31m+43,32m+43]
 \cup [39m+57,40m+57] \cup [40m+59,41m+59] .$\\

\bigskip\noindent
If $\bm{n \equiv 4 \pmod 8, n \geq 12}$: Let $m = \frac{n-12}{8}$.
All the arithmetic in this case will be in $\ZZ_{48m+73}$.
 The first eight columns are: \\

{\small
\resizebox{\linewidth}{!}{%
$A = \begin{bmatrix}
8m + 13 & 10m + 16 & 22m + 34 & -4m - 5 & 4m + 7 & -22m - 35 & -12m - 18 & -1 \\
4m + 6 &  8m + 11 & -4m - 8 & 22m + 33 & -14m - 22 & 4m + 10 & -2 & -20m - 30 \\
-12m - 19 & -18m - 27 & -18m - 26 & -18m - 28 & 10m + 15 & 18m + 25 & 12m + 20 & 20m + 31
\end{bmatrix}.$} }

For $0 \leq r \leq 2m$ define 
{\small
$$A_r = (-1)^r\begin{bmatrix}
(-16m + r - 23) & (-8m + 2r -12) & (14m - r + 21) & (4m - 2r + 3) \\
(8m + r + 14) & (-16m - r - 24) & (-10m - r - 17) & (18m + r + 29) \\
 (8m - 2r + 9) & (24m - r + 36) & (-4m + 2r - 4) & (-22m + r - 32)
 \end{bmatrix}.$$}

To construct $H$, begin with $A$ and add on the remaining $n-8$ columns by concatenating the $A_r$ arrays for each value of $r$ between $0$ and $2m$.  Let the reordering be 

\begin{center}
$R = (9,13,...,n-3;11,15,...,n-1;4,10,14,...,n-2;12,16,...,n,1,2,6,5, 7,8,3).$
\end{center}
Then
\begin{center}
\begin{tabular}{|rcl|}
\hline
$P_{1,1} $&=&$ [32m+50,33m+50] \cup [47m+73,48m+72] $\\
\hline
$P_{1,2} $&=&$ [46m+71,47m+71] \cup [33m+51,34m+50] $ \\
\hline
$P_{1,3} $&=&$ [34m+54,36m+54]_2 \cup [40m+66,42m+66]_2 $ \\
\hline
$P_{1,4} $&=&$ [38m+59,40m+57]_2 \cup [36m+56,38m+54]_2 $\\
&&$\cup \{38m+57,38m+58, 46m+70,8m+13,34m+51,26m+40,26m+39,0\}$ \\
\hline
\end{tabular}\end{center}

\noindent
Thus
$P_1 = \{0, 8m+13, 26m+39,26m+40\} \cup [32m+50,33m+50] \cup [33m+51,34m+50]  \cup \{34m+51\}
 \cup [34m+54,36m+54]_2 \cup [36m+56,38m+54]_2 \cup \{38m+57,38m+58\} \cup [38m+59,40m+57]_2 
\cup [40m+66,42m+66]_2 \cup \{46m+70\} \cup [46m+71,47m+71] \cup [47m+73,48m+72].  $

\begin{center}
\begin{tabular}{|rcl|}
\hline
$P_{2,1} $&=&$ [8m+14,9m+14] \cup [47m+73,48m+72] $\\
\hline
$P_{2,2} $&=&$ [9m+15,10m+14] \cup [46m+70,47m+70] $\\
\hline
$P_{2,3} $&=&$ [3m+6,4m+6] \cup [20m+30,21m+30] $ \\
\hline
$P_{2,4} $&=&$ [2m+6,3m+5] \cup [21m+35,22m+35] $\\
&&$\cup \{26m+41,34m+52,38m+62,24m+40,24m+38, 
4m+8,0\}$ \\
\hline
\end{tabular}\end{center}

\noindent
Thus
$P_2 = \{0\} \cup [2m+6,3m+5] \cup [3m+6,4m+6] \cup \{4m+8\} \cup [8m+14,9m+14] \cup [9m+15,10m+14]
\cup [20m+30,21m+30] \cup [21m+35,22m+35] \cup \{24m+38,24m+40,26m+41,34m+52,38m+62\}
\cup [46m+70,47m+70] \cup  [47m+73,48m+72]. $ 

\begin{center}
\begin{tabular}{|rcl|}
\hline
$P_{3,1} $&=&$ [2,2m]_2 \cup [6m+9,8m+7]_2 $\\
\hline
$P_{3,2} $&=&$ [2m+5,4m+5]_2 \cup [4m+9,6m+7]_2 $ \\
\hline
$P_{3,3} $&=&$ [9m+13,10m+13] \cup [34m+50,35m+50] $\\
\hline
$P_{3,4} $&=&$ [8m+13,9m+12] \cup [35m+54,36m+54] $\\ && $\cup \{24m+35,6m+8,24m+33,34m+48,46m+38,
18m+26,0\}$ \\
\hline
\end{tabular}\end{center}

\noindent
So 
$P_3 = \{0\} \cup [2,2m]_2 \cup [2m+5,4m+5]_2 \cup [4m+9,6m+7]_2 \cup \{6m+8\} \cup [6m+9,8m+7]_2
 \cup [8m+13,9m+12] \cup [9m+13,10m+13] \cup \{18m+26, 24m+33, 24m+35, 34m+48\}
 \cup [34m+50,35m+50] \cup [35m+54,36m+54] \cup \{46m+38\}.$

\bigskip
\noindent
If $\bm{n \equiv 5 \pmod 8, n \geq 5}$: Here $m = \frac{n-5}{8}$. 
All the arithmetic in this case will be in $\ZZ_{48m+31}$.
 The first five columns are: 
{\small
$$A = \begin{bmatrix}
8m + 6 & 10m + 7 & -16m - 10 & -4m - 4 & 4m + 1 \\
-16m - 9 & 8m + 5 & 4m + 2 & -18m - 11 & 18m + 13 \\
8m + 3 & -18m - 12 & 12m + 8 & 22m + 15 & -22m - 14
\end{bmatrix}.$$}
For each $0 \leq r \leq 2m-1$ define
{\small
$$A_r = (-1)^r\begin{bmatrix}
(-8m + 2r - 1) & (-14m + r - 8) & (16m + r + 11) & (4m - 2r - 1) \\
(16m - r + 8) & (4m - 2r) & (8m - 2r + 4) & (18m + r + 14) \\
(-8m - r - 7) & (10m + r + 8) & (-24m + r - 15) & (-22m + r - 13)
 \end{bmatrix}.$$} 
To construct $H$, begin with $A$ and add on the remaining $n-5$ columns by concatenating the $A_r$ arrays for each value of $r$ between $0$ and $2m-1$.  Let the reordering be 

\begin{center}
$R = (9,13,...,n;5,6,10,...,n-3;3,7,11,...,n-2;1,8,12,...,n-1,4,2).$
\end{center}

Then
\begin{center}
\begin{tabular}{|rcl|}
\hline
$P_{1,1} $&=&$ [2,2m]_2 \cup [2m+1,4m-1]_2$ \\
\hline
$P_{1,2} $&=&$ [46m+31,48m+29]_2 \cup [4m+1,6m+1]_2 $ \\
\hline
$P_{1,3} $&=&$ [35m+22,36m+22] \cup [22m+14,23m+13] $ \\
\hline
$P_{1,4} $&=&$ [42m+28,43m+28] \cup [11m+8,12m+7] \cup \{38m+24,0\}$ \\
\hline
\end{tabular}\end{center}

\noindent
Thus
$P_1 = \{0\} \cup [2,2m]_2 \cup [2m+1,4m-1]_2 \cup [4m+1,6m+1]_2 \cup [11m+8,12m+7] \cup [22m+14,23m+13]
 \cup [35m+22,36m+22] \cup \{38m+24\} \cup [42m+28,43m+28] \cup [46m+31,48m+29]_2.$

\begin{center}
\begin{tabular}{|rcl|}
\hline
$P_{2,1} $&=&$ [47m+31,48m+30] \cup [18m+14,19m+13] $ \\
\hline
$P_{2,2} $&=&$ [17m+13,18m+13] \cup [32m+22,33m+21] $\\
\hline
$P_{2,3} $&=&$ [22m+15,24m+15]_2 \cup [24m+17,26m+15]_2 $\\
\hline
$P_{2,4} $&=&$ [8m+6,10m+6]_2 \cup [14m+12,16m+10]_2 \cup \{40m+26,0\}$ \\
\hline
\end{tabular}\end{center}

\noindent
Thus
$P_2 = \{0\} \cup [8m+6,10m+6]_2 \cup [14m+12,16m+10]_2 \cup [17m+13,18m+13] \cup [18m+14,19m+13]
 \cup [22m+15,24m+15]_2 \cup [24m+17,26m+15]_2 \cup [32m+22,33m+21] \cup \{40m+26\}
 \cup [47m+31,48m+30].$

\begin{center}
\begin{tabular}{|rcl|}
\hline
$P_{3,1} $&=&$ [47m+31,48m+30] \cup [26m+18,27m+17] $ \\
\hline
$P_{3,2} $&=&$ [16m+11,17m+10] \cup [25m+17,26m+17] $\\
\hline
$P_{3,3} $&=&$ [2,m+1] \cup [37m+25,38m+25] $ \\
\hline
$P_{3,4} $&=&$ [21m+13,22m+12] \cup [44m+28,45m+28] \cup \{18m+12,0\}$ \\
\hline
\end{tabular}\end{center}

\noindent
So $P_3 = \{0\} \cup [2,m+1] \cup [16m+11,17m+10] \cup \{18m+12\} \cup [21m+13,22m+12]
\cup [25m+17,26m+17]\cup [26m+18,27m+17] \cup [37m+25,38m+25] \cup [44m+28,45m+28]
 \cup [47m+31,48m+30]. $

\bigskip
\noindent
If $\bm{n \equiv 6 \pmod 8, n \geq 6}$: In this case, $m = \frac{n-6}{8}$.
All the arithmetic in this case will be in $\ZZ_{48m+37}$.
The first six columns are:
{\small
$$A = \begin{bmatrix}
24m + 18 & -16m - 13 & -1 & 8m + 4 & -4m - 3 & -8m - 5 \\
2 & 8m + 6 & -10m - 8 & -20m - 14 & -16m - 12 & -12m - 11 \\
24m + 17 & 8m + 7 & 10m + 9 & 12m + 10 & 20m + 15 & 20m + 16
\end{bmatrix}.$$}
For each $0 \leq r \leq 2m-1$ define
{\small
$$A_r = (-1)^r\begin{bmatrix}
(-8m + 2r - 3) & (-4m + 2r - 1) & (-4m + 2r - 2) & (8m - 2r + 2) \\
(16m - r + 11) & (-10m - r - 10) & (22m - r + 16) & (16m + r + 14) \\
(-8m - r - 8) & (14m - r + 11) & (-18m - r - 14) & (-24m + r - 16)
 \end{bmatrix}.$$}

To construct $H$, begin with $A$ and add on the remaining $n-6$ columns by concatenating the $A_r$ arrays for each value of $r$ between $0$ and $2m-1$.  Let the reordering be 

\begin{center}
$R = (10,14,...,n;2,9,13,...,n-1;4,7,11,...,n-3;1,8,12,...,n-2,5,3,6).$
\end{center}

Then
\begin{center}
\begin{tabular}{|rcl|}
\hline
$P_{1,1} $&=&$ [2,2m]_2 \cup [6m+4,8m+2]_2 $ \\
\hline
$P_{1,2} $&=&$ [30m+22,32m+20]_2 \cup [32m+24,34m+24]_2 $ \\
\hline
$P_{1,3} $&=&$ [32m+25,34m+23]_2 \cup [38m+28,40m+28]_2 $ \\
\hline
$P_{1,4} $&=&$ [10m+8,12m+6]_2 \cup [12m+9,14m+9]_2 \cup \{8m+6, 8m+5,0\}$ \\
\hline
\end{tabular}\end{center}

\noindent
Thus $P_1 = \{0\} \cup [2,2m]_2 \cup [6m+4,8m+2]_2 \cup \{8m+5, 8m+6\} \cup [10m+8,12m+6]_2
\cup [12m+9,14m+9]_2 \cup [30m+22,32m+20]_2 \cup [32m+24,34m+24]_2 \cup [32m+25,34m+23]_2
 \cup [38m+28,40m+28]_2. $

\begin{center}
\begin{tabular}{|rcl|}
\hline
$P_{2,1} $&=&$ [16m+14,7m+13] \cup [47m+37,48m+36] $\\
\hline
$P_{2,2} $&=&$ [7m+6,8m+6] \cup [28m+23,29m+22] $\\
\hline
$P_{2,3} $&=&$ [36m+29,37m+29] \cup [3m+4,4m+3]$\\
\hline
$P_{2,4} $&=&$ [26m+22,27m+21] \cup [37m+31,38m+31] \cup \{22m+19, 12m+11,0\} $ \\
\hline
\end{tabular}\end{center}

\noindent
Thus $P_2 = \{0\} \cup [3m+4,4m+3] \cup [7m+6,8m+6] \cup \{12m+11\} \cup [16m+14,7m+13]  \cup \{22m+19\}
 \cup [26m+22,27m+21] \cup [28m+23,29m+22]  \cup [36m+29,37m+29] \cup [37m+31,38m+31]
 \cup [47m+37,48m+36].  $

\begin{center}
\begin{tabular}{|rcl|}
\hline
$P_{3,1} $&=&$ [24m+21,25m+20] \cup [47m+37,48m+36] $ \\
\hline
$P_{3,2} $&=&$ [36m+31,37m+30] \cup [7m+7,8m+7] $ \\
\hline
$P_{3,3} $&=&$ [20m+17,21m+17] \cup [11m+10,12m+9] $ \\
\hline
$P_{3,4} $&=&$ [10m+9,11m+8] \cup [45m+34,46m+34] \cup \{18m+12,28m+21,0\}$ \\
\hline
\end{tabular}\end{center}

\noindent
So $P_3 = \{0\} \cup [7m+7,8m+7] \cup [10m+9,11m+8] \cup [11m+10,12m+9] \cup \{18m+12\}
 \cup  [20m+17,21m+17] \cup [24m+21,25m+20] \cup \{28m+21\} \cup [36m+31,37m+30]
 \cup [45m+34,46m+34] \cup [47m+37,48m+36]. $

\bigskip
\noindent
If $\bm{n \equiv 7 \pmod 8, n \geq 7}$: Now let $m = \frac{n-7}{8}$.
All the arithmetic in this case will be in $\ZZ_{48m+43}$.
The first seven columns are: 
{\small
$$A = \begin{bmatrix}
24m + 21 & 16m + 15 & 4m + 3 & -4m - 4 & -20m - 18 & -12m - 11 & -8m - 6 \\
2 & -8m - 8 & -12m - 12 & 14m + 14 & 1 & 20m + 16 & -14m - 13 \\
24m + 20 & -8m - 7 & 8m + 9 & -10m - 10 &  20m + 17 & -8m - 5 & 22m + 19
\end{bmatrix}.$$} 
For each $0 \leq r \leq 2m-1$ define
{\small
$$A_r = (-1)^r \begin{bmatrix}
(-16m + r - 14) & (-8m + 2r - 3) & (-18m - r - 16) & (4m - 2r + 1) \\
(8m + r + 10) & (-16m - r - 16) & (22m - r + 18) & (10m + r + 11) \\
(8m - 2r + 4) & (24m - r + 19) & (-4m + 2r - 2) & (-14m + r - 12)
\end{bmatrix}.$$}
To construct $H$, begin with $A$ and add on the remaining $n-7$ columns by concatenating the $A_r$ arrays for each value of $r$ between $0$ and $2m-1$.  Let the reordering be 

\begin{center}
$R = (10,14,...,n-1;2,8,12,...,n-3;6,11,15,...,n;7,9,13,...,n-2;4,3 ,1,5).$
\end{center}

Then
\begin{center}
\begin{tabular}{|rcl|}
\hline
$P_{1,1} $&=&$ [1,m] \cup [29m+28,30m+27] $\\
\hline
$P_{1,2} $&=&$ [m+1,2m] \cup [16m+15,17m+15]$ \\
\hline
$P_{1,3} $&=&$ [6m+7,8m+]_2 \cup [4m+4,6m+4]_2 $ \\
\hline
$P_{1,4} $&=&$ [38m+38,40m+36]_2 \cup [44m+41,46m+41]_2 \cup \{40m+37,44m+40,20m+18,0\}$ \\
\hline
\end{tabular}\end{center}

\noindent
Thus $P_1 = \{0\} \cup [1,m] \cup [m+1,2m] \cup [4m+4,6m+4]_2 \cup [6m+7,8m+]_2 \cup [16m+15,17m+15]
 \cup \{20m+18\} \cup [29m+28,30m+27] \cup [38m+38,40m+36]_2 \cup \{40m+37,44m+40\}
\cup [44m+41,46m+41]_2. $

\begin{center}
\begin{tabular}{|rcl|}
\hline
$P_{2,1} $&=&$ [1,m] \cup [21m+19,22m+18]$ \\
\hline
$P_{2,2} $&=&$ [m+2,2m+1] \cup [40m+35,41m+35]$ \\
\hline
$P_{2,3} $&=&$ [11m+8,12m+8] \cup [22m+19,23m+18] $\\
\hline
$P_{2,4} $&=&$ [45m+38,46m+38] \cup [28m+23,29m+22] \cup \{12m+9,48m+40,48m+42,0\}$ \\
\hline
\end{tabular}\end{center}

\noindent
Thus
$P_2 = \{0\} \cup [1,m] \cup [m+2,2m+1] \cup [11m+8,12m+8] \cup \{12m+9\} \cup [21m+19,22m+18]
\cup [22m+19,23m+18] \cup [28m+23,29m+22] \cup [40m+35,41m+35] \cup [45m+38,46m+38]
 \cup \{48m+40,48m+42\}.$ 

\begin{center}
\begin{tabular}{|rcl|}
\hline
$P_{3,1} $&=&$ [46m+43,48m+41]_2 \cup [44m+41,46m+39]_2 $ \\
\hline
$P_{3,2} $&=&$ [44m+42,46m+40]_2 \cup [38m+36,40m+36]_2 $ \\
\hline
$P_{3,3} $&=&$ [18m+19,19m+18] \cup [31m+31,32m+31] $ \\
\hline
$P_{3,4} $&=&$ [5m+7,6m+7] \cup [28m+27,29m+26] \cup \{44m+40,4m+6,28m+26,0\}$ \\
\hline
\end{tabular}\end{center}

\noindent
So $P_3 = \{0\} \cup \{4m+6\} \cup [5m+7,6m+7] \cup [18m+19,19m+18] \cup \{28m+26\} \cup [28m+27,29m+26]
\cup [31m+31,32m+31]  \cup [38m+36,40m+36]_2 \cup \{44m+40\} \cup [44m+41,46m+39]_2
\cup [44m+42,46m+40]_2 \cup[46m+43,48m+41]_2.  $
\end{proof}

Now that we have established that for every $n\geq 3$ there exists  simple row and column orderings for each  Heffter array  $H(3,n)$ from \cite{Ain} we can prove the  main result of this paper.

\begin{thm} For every  $n \geq 3$,
there exists an orientable biembedding of $K_{6n+1}$  such that every edge is on a 3-cycle and a simple $n$-cycle, or equivalently, for every  $n \geq 3$,
there exists an orientable biembedding  of a Steiner triple system and a simple $n-$cycle system, both on $6n+1$ points.
Furthermore, each of the two cycle systems is cyclic modulo $6n+1$.
\end{thm}

\begin{proof}
By Theorem \ref{reord2}, given any $n \in \mathbb Z$ with $n \geq 3$, there exists a $3 \times n$ simple Heffter array.   By Corollary \ref{heffter-biembed.3xn} it follows that
 there exists an embedding of $K_{6n+1}$ on an orientable surface such that every edge is on a simple cycle face of size $3$ and a simple cycle face of size $n$.  From Proposition \ref{prop2.1} we have that each cycle system is cyclic.
\end{proof}

It is interesting to note on which orientable surface we are biembedding. Euler's formula, $V - E + F = 2 - 2g$, can be used to determine the genus of the surface. It is easy to compute that for $K_{6n+1}$; the number of vertices is $V = 6n + 1$, the number of edges is $E = {6n+1 \choose 2}$, and the number of faces is $F = {6n+1 \choose 2} (1/3 + 1/n)$. Substituting these values into Euler's formula we get the following proposition.

\begin{prop}
\label{euler}
For  $n \geq 3$, $K_{6n+1}$ can be biembedded such that every edge is on an $n$-cycle and a 3-cycle on the orientable surface with genus 
$$g =1 - 1/2\Big[6n + 1 + {6n+1 \choose 2}(1/3 + 1/n -1)\Big].$$ 
\end{prop}

\begin{example}
Letting  $n = 5$ in  Proposition \ref{euler} above we get
$$g = 1 - 1/2 \Big[31 + {31 \choose 2 }(1/3 + 1/5 -1) \Big] = 1 - 1/2(31 + 465(-7/15)) = 94.$$
So $K_{31}$ can be embedded on an orientable surface with genus 94 such that every edge is on both a 3-cycle and a 5-cycle.
\end{example}

\section{$5 \times n$ Heffter Arrays} \label{section4}
An obvious continuation of the $3 \times n$ result is to ask whether we can use $5 \times n$ Heffter arrays to biembed $K_{10n + 1}$ such that every edge is on both a $5-$ cycle and an $n-$cycle.  Via Theorem \ref{heffter-biembed.thm}, since 5 is odd we have the following corollary.

\begin{corollary}\label{5xn}
If there exists a simple Heffter array $H(5,n)$, then there exists an orientable embedding of $K_{10n+1}$  such that every edge is on a simple cycle face of size $5$ and a simple cycle face of size $n$ 
\end{corollary}

As was done in the case of the $H(3,n)$ we start with an $H(5,n)$ and rearrange it so that the resulting $H(5,n)$
is simple.  All of the necessary  $H(5,n)$ exist via the following theorem.

\begin{thm}\label {5xn,exist}\cite{Ain}
There exists a $5 \times n$ Heffter array for every $n \geq 3$.
\end{thm}

Considering the $H(5,n)$ from \cite{Ain} one can easily verify that the columns are simple in the standard ordering and so again we only need to reorder to rows. However, unlike the $3 \times n$ case, we were unable to do this in general.  In order to obtain a partial result we found reorderings for $3\leq n \leq 100$ using a computer (this was not difficult).
 We again use a single permutation for every row, which in fact reorders the Heffter array by permuting the columns as units. We do not list the computed permuations here; the interested reader can find them in Appendix A of \cite{M15}.

\begin{prop}\label {5xn,100} \cite{M15}
There exists a simple $5 \times n$ Heffter array for every $3 \leq n \leq 100$.
\end{prop}

So via Theorem \ref{5xn,exist} and Proposition  \ref{5xn,100} we have the main result of this section.

\begin{thm}
For every $3 \leq n \leq 100$, there exists an orientable biembedding of $K_{10n+1}$ such that every edge is on a simple cycle face of size $5$ and a simple cycle face of size $n$.
\end{thm}

\section{Conclusion}
In this paper we have shown that for every  $n \geq 3$,
there exists an orientable biembedding  of a Steiner triple system and a simple $n-$cycle system, both on $6n+1$ points.
This paper is the first to exploit the connection which Archdeacon found between Heffter arrays and biembeddings the complete graph on a surface to explicitly biembed a class of graphs.  We considered only the case of biembeddings arising from the existence of Heffter arrays $H(3,n)$ and $H(5,n)$.  In \cite{Ain} Heffter arrays $H(m,n)$ are given for all $m,n \geq 3$.  Hence there is certainly an opportunity to use these for the biembedding problem (if one can find simple orderings of the rows and columns of these arrays).  In addition, a more general definition of Heffter arrays from \cite{A14} leads to biembeddings of other complete graphs (in addition to $K_{2mn+1}$) as well as to biembeddings on  nonorientable surfaces.  More Heffter arrays which could possibly be used to construct biembeddings can be found in the paper \cite{Aincomp}.

\bibliographystyle{unsrt}

\end{document}